\documentclass[A4,11pt]{article}
\usepackage{amsfonts}
\usepackage{amsmath,amsthm}
\usepackage{amssymb,amscd}

\usepackage{cite}
\usepackage{esint}

\usepackage{multicol}

\usepackage{kbordermatrix}

\makeatletter
\newcommand{\rmnum}[1]{\romannumeral #1}
\newcommand{\Rmnum}[1]{\expandafter\@slowromancap\romannumeral #1@}
\makeatother
\def\Xint#1{\mathchoice
{\XXint\displaystyle\textstyle{#1}}%
{\XXint\textstyle\scriptstyle{#1}}%
{\XXint\scriptstyle\scriptscriptstyle{#1}}%
{\XXint\scriptscriptstyle\scriptscriptstyle{#1}}%
\!\int}
\def\XXint#1#2#3{{\setbox0=\hbox{$#1{#2#3}{\int}$}
\vcenter{\hbox{$#2#3$}}\kern-.5\wd0}}

\def\dashint{\Xint-}

\renewcommand{\kbldelim}{(}
\renewcommand{\kbrdelim}{)}

\def\Xint#1{\mathchoice
{\XXint\displaystyle\textstyle{#1}}%
{\XXint\textstyle\scriptstyle{#1}}%
{\XXint\scriptstyle\scriptscriptstyle{#1}}%
{\XXint\scriptscriptstyle\scriptscriptstyle{#1}}%
\!\int}
\def\XXint#1#2#3{{\setbox0=\hbox{$#1{#2#3}{\int}$}
\vcenter{\hbox{$#2#3$}}\kern-.5\wd0}}

\def\dashint{\Xint-}

\theoremstyle{definition}
\newtheorem{theorem}{Theorem}[section]

\newtheorem{corollary}[theorem]{Corollary}
\newtheorem{lemma}[theorem]{Lemma}

\newtheorem{remark}{Remark}[section]

\begin{document}
\title{Some conditional probabilities in the TASEP with second class particles}
\author{\textbf{Eunghyun Lee\footnote{eunghyun.lee@nu.edu.kz}}\\ {\text{Department of Mathematics, Nazarbayev University}}
                                         \date{}   \\ {\text{Kazakhstan}}   }

\date{}
\maketitle
\begin{abstract}
\noindent In this paper we consider the  TASEP with second class particles which consists of $k$ first class particles and $N-k$ second class particles. We assume that all first class particles  are initially  located to the left of the leftmost second class particle. Under this assumption, we find the probability that the first class particles  are at $x,x+1,\cdots, x+k-1$  and these positions  are still to the left of the leftmost second class particle at time $t$. If we additionally assume that the initial positions of the particles are $1,\cdots, N$, that is, step initial condition, then the formula of the probability does not depend on $k$ and is very similar to a formula for the TASEP (without second class particles) with step initial condition.
\end{abstract}
\section{Introduction and the statements of results}
The one-dimensional totally asymmetric simple exclusion process (TASEP) is one of the simplest particle models that is \textit{exactly solvable} by the standard technique, the Bethe Ansatz. The Bethe Ansatz has been widely used to study the TASEP, its generalizations, and its variants \cite{Corwin-2014,Korhonen-Lee-2014,Lee-2012,Povolotsky-2013,Nagao-Sasamoto-2004,Schutz-1997,Tracy-Widom-2008,Wang-Waugh-2016}. In the TASEP with second class particles there are two species of particles which are labeled  1 and 2, respectively, on the integer lattice $\mathbb{Z}$. Basically, each particle follows the rules of the TASEP. That is, each particle waits exponential time and then jumps to the next right site if the site is empty. If the target site is occupied by a particle of the same species, the jump is prohibited. However, if a particle of species 2 tries to jump to the next right site and the  site is occupied by a particle of species 1, then the particle of species 2 can jump to the next right site by interchanging the positions of the two particles. But, if a particle of species 1 tries to jump to the next right site occupied by a particle of species 2, the jump is prohibited. That is, particles of species 2 have priority  over particles of species 1. For this reason,  particles of species 2 are called first class particles and particles of species 1 are called second class particles.

The TASEP with second class particles has been studied in various directions, but to the best of the author's knowledge, there are only a few works  on the TASEP with second class particles using the Bethe Ansatz technique to find some explicit formulas and related asymptotics \cite{Chatterjee-Schutz-2010,Lee-2017}.  Chatterjee and Sch\"{u}tz studied the transition probabilities of a finite system of the TASEP with second class particles. In particular, they showed that if there is no change in the order of particles, the transition probabilities can be  written in determinantal forms \cite{Chatterjee-Schutz-2010}. Inspired by this fact, the author considered a system with 1 first class particle and $N-1$ second class particles where the leftmost particle at $t=0$ is the first class particle, and  found the probability that the leftmost particle is at $x$ at time $t$ under the assumption that there is no change in the order of particles up to time $t$ \cite{Lee-2017}. Moreover, it was shown that this probability can also be written in a determinantal form  if the initial positions of the particles are given by  step initial condition \cite{Lee-2017}. This determinantal formula is very similar to a formula of the probability for the leftmost particle's position in the TASEP with step initial condition. It is well known that there appears the Tracy-Widom distribution in an asymptotical behaviour of the probability distribution of the leftmost particle's position in a large system of the TASEP with step initial condition \cite{Johansson-2000,Nagao-Sasamoto-2004,Tracy-Widom-1994}. Hence, it is expected that a similar asymptotic behaviour may appear in the TASEP with second class particles.
\\ \\
In this paper, we extend the results in \cite{Lee-2017} to the case where there are $k$ first class particles in the system. Let us introduce some notations. In most cases, we will follow the notations in \cite{Lee-2017}. We denote a state of the process by pairs  $(X,\pi)$ where $X=(x_1,\cdots,x_N) \in \mathbb{W}^N$ with
\begin{equation*}
\mathbb{W}^N = \{(x_1,\cdots, x_N): (x_1,\cdots, x_N) \in \mathbb{Z}^N ~\textrm{and}~ x_1 < \cdots < x_N\}
\end{equation*}
and $\pi=(\pi_1\cdots\pi_N)$ is a finite sequence of length $N$ whose elements are 1 or 2. A state
\begin{equation*}
\big((x_1,\cdots, x_N),(\pi_1\cdots\pi_N)\big)
\end{equation*}
implies that the position of the $i^{\textrm{th}}$ particle from the left is $x_i$ and the $i^{\textrm{th}}$ particle is labeled $\pi_i$. The parentheses in the expressions $(x_1,\cdots,x_N)$ and $(\pi_1\cdots\pi_N)$ are sometimes omitted for convenience, and when we need to specify a state at time $t$, we write
\begin{equation*}
\big((x_1,\cdots, x_N),(\pi_1\cdots\pi_N);t\big).
\end{equation*}
For a fixed $1 \leq k \leq N$, let  $\nu^{(k)} = \big(\nu^{(k)}_j\big)$ be a finite sequence of numbers 1 or 2 with length $N$ such that $\nu^{(k)}_1 = \nu^{(k)}_2 = \cdots = \nu^{(k)}_k = 2$ and  $\nu^{(k)}_{k+1}= \cdots =\nu^{(k)}_{N} = 1$. Let $\nu^{(0)}$  be a sequence $(1,\cdots,1)$. For a fixed $1 \leq k \leq N$, let
\begin{equation*}
E_{t,k,x} = \{(X,\nu^{(k)};t): x_1=x, x_2=x+1,\cdots,x_k=x+(k-1)\}
\end{equation*}
and
\begin{equation*}
E_{t,0,x} = \{(X,\nu^{(0)};t): x_1\geq x\}.
\end{equation*}
 Let $\mathbb{P}_{(Y,\nu^{(k)})}$ be the probability law of the process starting from $(Y,\nu^{(k)})$, and let  $\dashint$ imply $\frac{1}{2\pi i}\int$. Throughout this paper, we will follow a convention for products
\begin{equation*}
\prod_{i=m}^n A_i := \begin{cases}A_m A_{m+1}\cdots A_n&~~\textrm{if}~n\geq m\\
1&~~\textrm{if}~n< m,
\end{cases}
\end{equation*} and we define $\varepsilon(\xi_i) = 1/\xi_i - 1$.
\begin{theorem}\label{205-am-618}
Let $C$ be a counterclockwise circle centered at the origin with radius less than 1. Then, for each $k=0,\cdots, N$,
\begin{equation}\label{225-am-618}
\mathbb{P}_{(Y,\nu^{(k)})}(E_{t,k,x}) = \dashint_C\cdots\dashint_C\prod_{i=1}^k(1-\xi_i)\prod_{1\leq i<j\leq N}\frac{\xi_j-\xi_i}{1-\xi_i}\prod_{i=1}^N\frac{1}{1-\xi_i}\prod_{i=1}^N\bigg(\xi_i^{x-y_i-1}e^{\varepsilon(\xi_i)t}\bigg)d\xi_1\cdots d\xi_N.
\end{equation}
\end{theorem}
In (\ref{225-am-618}), if $k=0$, the result is for the leftmost particle's position in the TASEP \cite{Johansson-2000,Nagao-Sasamoto-2004,Tracy-Widom-2008}. If $k=1$, then the result is equal to Theorem 1.1 in \cite{Lee-2017}. Now, we provide a formula for  step initial condition.
\begin{theorem}\label{1143-am-629}
For $1\leq k \leq N$, if $Y = (1,2,\cdots, N)$, then
\begin{equation}\label{534-pm-517}
\begin{aligned}
 \mathbb{P}_{(Y,\nu^{(k)})}(E_{t,k,x})~ =~&   \frac{(-1)^{N(N-1)/2}}{N!} \dashint_C\cdots\dashint_C\prod_{1\leq i<j\leq N}(\xi_j-\xi_i)^2\prod_{i=1}^N\frac{1}{(\xi_i-1)^{N-1}}\\
& \hspace{3.5cm}\times\prod_{i=1}^N \Big(\xi_i^{x-N-1}e^{\varepsilon(\xi_i) t}\Big)d\xi_1\cdots d\xi_N.
\end{aligned}
\end{equation}
\end{theorem}
The formula (\ref{534-pm-517}) is comparable to the formula for the case of $k=0$, that is, the probability distribution of the leftmost particle's position in the TASEP \cite{Johansson-2000,Nagao-Sasamoto-2004,Tracy-Widom-2008},
\begin{equation*}\label{534-pm-619}
\begin{aligned}
 \mathbb{P}_{(Y,\nu^{(0)})}(E_{t,0,x})~ =~&   \frac{(-1)^{N(N-1)/2}}{N!} \dashint_C\cdots\dashint_C\prod_{1\leq i<j\leq N}(\xi_j-\xi_i)^2\prod_{i=1}^N\frac{1}{(\xi_i-1)^{N}}\\
 & \hspace{3.5cm}\prod_{i=1}^N \Big(\xi_i^{x-N-1}e^{\varepsilon(\xi_i) t}\Big)d\xi_1\cdots d\xi_N.
\end{aligned}
\end{equation*}
\begin{remark}
Tracy and Widom posted a paper on the \textit{block probabilities} (named first in their paper \cite{Tracy-Widom-2017}) of the asymmetric simple exclusion process (ASEP) around a week after the first version of this paper was posted. According to Tracy and Widom's definition, a \textit{block} of length $k$ means a contiguous block of $k$ particles. The \textit{block probabilities} in the ASEP mean   $\mathbb{P}_{(Y,(22\cdots 2))}(E_{t,k,x}^{m})$ where
\begin{equation*}
E_{t,k,x}^m = \{(X,(22\cdots 2);t): x_m=x, x_{m+1}=x+1,\cdots,x_{m+k-1}=x+k-1\}.
\end{equation*}
Hence, the formula (\ref{534-pm-517}) with $k=N$ is the \textit{block probability} for the block of length $N$ in the TASEP (with a single species) with step initial condition. Indeed, it is confirmed that (\ref{534-pm-517}) with $k=N$ is equal to the formula  in Theorem 1  with $p=1$, $m=1$, and $L=N$ in \cite{Tracy-Widom-2017}.
\end{remark}
\begin{remark}
It is noticeable that the formula (\ref{534-pm-517}) has no dependency on $k$. That is, the probability that all $k$ first class particles form a \textit{block} which begins at $x$ and all second class particles are to the right of the block at time $t$ does not depend on the lengh of the \textit{block} but does depend on the position of the beginning of the \textit{block}. On the other hand, Tracy and Widom's formula of the ASEP without second class particles depends on the length of the \textit{block}  {\cite[Theorem 3]{Tracy-Widom-2017}}.
\end{remark}
\textbf{Organization of the paper}\\\\
The rest of this paper is organized as follows. In Section 2, we provide some preliminary results for the proofs of our main results. In Section 3, we prove Theorem \ref{205-am-618} and Theorem \ref{1143-am-629}.
\section{Preliminary results}
 The probability that the system will be at state $(X,\pi)$ at time $t$ given that the system started from state $(Y,\nu)$ at time $t=0$, is denoted by $P_{(Y,\nu)}(X,\pi;t)$. The transition probabilities, $P_{(Y,\nu)}(X,\pi;t)$ are the elements of infinite matrices $\mathbf{P}(t)$ which form a semigroup $\{\mathbf{P}(t): t \geq 0\}$. A submatrix of $\mathbf{P}(t)$, denoted by $\mathbf{P}_Y(X;t)$, consists of $P_{(Y,\nu)}(X,\pi;t)$ for fixed $X$ and $Y$. Hence, $\mathbf{P}_Y(X;t)$ is a $2^N \times 2^N$ matrix. The elements $P_{(Y,\nu)}(X,\pi;t)$ are listed in the reverse lexicographic order of $\nu$ and $\pi$ from left to right and from top to bottom, respectively, in $\mathbf{P}_Y(X;t)$. (See (6) in \cite{Lee-2017} for an example.) The formula of $\mathbf{P}_Y(X;t)$ is
\begin{equation}\label{1223-am-519}
\mathbf{P}_Y(X;t) = \dashint_C\cdots \dashint_C\sum_{\sigma\in {S}_N}\mathbf{A}_{\sigma}\prod_{i=1}^N\Big(\xi_{\sigma(i)}^{x_i-y_{\sigma(i)}-1}e^{\varepsilon(\xi_i) t}\Big) d\xi_1\cdots d\xi_N
\end{equation}
where $C$ is a counterclockwise circle centered at the origin with its radius less than 1 and $\mathbf{A}_{\sigma}$ is a certain $2^N\times 2^N$ matrix, which will be introduced in more detail in the subsection \ref{903}  \cite{Chatterjee-Schutz-2010,Lee-2017,Tracy-Widom-2013}.
\begin{remark}
 The explicit formulas of the elements of  $\mathbf{A}_{\sigma}$ in (\ref{1223-am-519}) were not given in \cite{Chatterjee-Schutz-2010,Tracy-Widom-2013} but the explicit formula of the $(2^{N-1}+1,2^{N-1}+1)^{\textrm{th}}$ element of $\mathbf{A}_{\sigma}$ which is for $P_{(Y,\nu^{(1)})}(X,\nu^{(1)};t)$, was given in \cite{Lee-2017}. In general, the matrices $\mathbf{A}_{\sigma}$ are the products of some matrices. To the best of the author's knowledge, there are no known ways of identifying each element of $\mathbf{A}_{\sigma}$ without directly computing the matrix products except some special cases.
  \end{remark}
 \subsection{Explicit formulas of $P_{(Y,\nu^{(k)})}(X,\nu^{(k)};t)$}\label{903}
 First, we state the notations used in \cite{Lee-2017}. Let $T_l$ be the simple transposition in the symmetric group  $S_N$ such that for any $\sigma \in S_N$, $\big(T_l\sigma\big)(l) = \sigma(l+1)$,  $\big(T_l\sigma\big)(l+1) = \sigma(l)$, and $\big(T_l\sigma\big)(j) = \sigma(j)$ if $j\neq l,l+1$, for each $l=1,2,\cdots, N-1$.
If $T_l$ interchanges $\alpha$ and $\beta$, that is, $T_l(\cdots \alpha\beta\cdots) = (\cdots\beta\alpha\cdots)$,
we define a $2^N\times 2^N$ matrix,
\begin{equation} \label{1243-am-519}
\mathbf{T}_{l}=\mathbf{T}_{l}({\alpha},{\beta})=
 \mathbf{I}_2^{\otimes(l-1)} \otimes \mathbf{S}_{\beta\alpha} \otimes  \mathbf{I}_2^{\otimes(N-l-1)}, ~ (l=1,\cdots,N-1)
\end{equation}
where
\begin{equation}
\mathbf{S}_{\beta\alpha} =  \left[
                                                                                                     \begin{array}{cccc}
                                                                                                       -\frac{1-\xi_{\beta}}{1-\xi_{\alpha}} & 0 & 0 & 0 \\
                                                                                                       0 & -\frac{1-\xi_{\beta}}{1-\xi_{\alpha}} & \frac{\xi_{\beta}-\xi_{\alpha}}{1-\xi_{\alpha}} & 0 \\
                                                                                                       0 & 0 & -1 & 0 \\
                                                                                                       0 & 0 & 0 & -\frac{1-\xi_{\beta}}{1-\xi_{\alpha}} \\
                                                                                                     \end{array}
                                                                                                   \right]
                                                                                                 \label{546pm511}
\end{equation}
and $\otimes$ implies the tensor product of matrices.
Since $T_1,\cdots,T_{N-1}$ generate $S_N$, for any given $\sigma \in S_N$, there are  $T_{a_1}, \cdots, T_{a_n}$ such that
\begin{equation}
\sigma = T_{a_n}\cdots T_{a_1} \label{411-pm-614}
\end{equation}
where $a_1,\cdots, a_n$ are some integers in $ \{1,\cdots, N-1\}$. It is known that $\mathbf{T}_i$'s satisfy the consistency conditions  \cite{Chatterjee-Schutz-2010}. For any $\sigma \in S_N$ written as in (\ref{411-pm-614}), $\mathbf{A}_{\sigma}$ in (\ref{1223-am-519}) is given by
\begin{equation}
\mathbf{A}_{\sigma} = \mathbf{T}_{a_n}\cdots \mathbf{T}_{a_1}  \label{417-pn-614}
\end{equation}
\cite{Chatterjee-Schutz-2010,Lee-2017,Tracy-Widom-2013}. For simplicity in notations, let
\begin{equation*}
h_k = \begin{cases}
 2^{N-1} + 2^{N-2} + \cdots + 2^{N-k}+1&~~\textrm{for $k=1,\cdots,N$},\\
1&~~\textrm{for $k=0$.}
\end{cases}
\end{equation*}
Then, we observe that $P_{(Y,\nu^{(k)})}(X,\nu^{(k)};t)$ is the $(h_k,h_k)^{
\textrm{th}}$ element of $\mathbf{P}_Y(X;t)$. We write $[\mathbf{A}]_{i,j}$ for the $(i,j)^{\textrm{th}}$ element of a matrix $\mathbf{A}$ and  $\mathbf{I}_n$ for the $n \times n$ identity matrix, and we wish to find an explicit formula of  $[\mathbf{A}_{\sigma}]_{h_k,h_k}$ in (\ref{1223-am-519}).
\begin{lemma}\label{340am629}
\begin{itemize}
\item [(\rmnum{1})] If $\mathbf{A}$ is an $n \times n$ matrix, then
 \begin{equation}\label{833-pm-617}
\big[\mathbf{A}\big]_{i,i} = \big[\mathbf{A}\otimes \mathbf{I}_2\big]_{2i,2i} = \big[\mathbf{A}\otimes \mathbf{I}_2\big]_{2i-1,2i-1}
\end{equation}
and
 \begin{equation}\label{834-pm-617}
\big[\mathbf{A}\big]_{i,i} = \big[ \mathbf{I}_2 \otimes \mathbf{A}\big]_{i,i} = \big[ \mathbf{I}_2 \otimes \mathbf{A}\big]_{i+n,i+n}.
\end{equation}
\item [(\rmnum{2})] For  $l=1,2,\cdots, N-1$ and $k=0,\cdots, N$, we have
\begin{equation}\label{719-pm-617}
[\mathbf{T}_l]_{h_k,h_k} = [\mathbf{T}_l(\alpha,\beta)]_{h_k,h_k} = \begin{cases}
-1 & ~~\textrm{if}~~l=k,\\
-\frac{1-\xi_{\beta}}{1-\xi_{\alpha}}&  ~~\textrm{if}~~l\neq k.
\end{cases}
\end{equation}
\end{itemize}
\end{lemma}
\begin{proof}
\begin{itemize}
  \item [(\rmnum{1})] It is easy to verify the statement from the forms of $\mathbf{A}\otimes \mathbf{I}_2$ and $ \mathbf{I}_2 \otimes \mathbf{A}$ for any $\mathbf{A}$.
  \item [(\rmnum{2})] We prove by induction on $N$. When $N=2$, the statement is obvious from (\ref{546pm511}). For the induction step, let us write
  \begin{equation*} \label{714-am-617}
\mathbf{T}_{l}^{(N)}= \mathbf{I}_2^{\otimes(l-1)} \otimes \mathbf{S}_{\beta\alpha} \otimes  \mathbf{I}_2^{\otimes(N-l-1)}, ~(l=1,\cdots,N-1),
  \end{equation*}
  and
 \begin{equation*}
h_k^{(N)} = \begin{cases}
 2^{N-1} + 2^{N-2} + \cdots + 2^{N-k}+1&~~\textrm{for $k=1,\cdots,N$},\\
1&~~\textrm{for $k=0$.}
\end{cases}
\end{equation*}
 Suppose that the statement is true for $N-1$, that is,
  \begin{equation}\label{720-pm-617}
[\mathbf{T}_l^{(N-1)}]_{h_k^{(N-1)},h_k^{(N-1)}} = \begin{cases}
-1 & ~~\textrm{if}~~l=k,\\
-\frac{1-\xi_{\beta}}{1-\xi_{\alpha}}&  ~~\textrm{if}~~l\neq k
\end{cases}
\end{equation}
for each $l =1,\cdots,N-2$ and for each $k=0,\cdots,N-1$. For each $l=1,\cdots,N-1$, $\mathbf{T}_l^{(N)}$ can be written as either
\begin{equation*}
\mathbf{T}_l^{(N)} = \mathbf{T}_l^{(N-1)} \otimes \mathbf{I}_2,~~(l=1,\cdots,N-2),
\end{equation*}
or
\begin{equation*}
 \mathbf{T}_l^{(N)} = \mathbf{I}_2 \otimes \mathbf{T}_{l-1}^{(N-1)},~~(l=2,\cdots,N-1).
\end{equation*}
If $\mathbf{T}_l^{(N)} = \mathbf{T}_l^{(N-1)} \otimes \mathbf{I}_2$, then for each $k=0,\cdots, N-1$ and for each $l=1,\cdots,N-2$,
\begin{equation*}
\big[\mathbf{T}_l^{(N)}\big]_{h_k^{(N)},h_k^{(N)}}= \big[\mathbf{T}_l^{(N-1)} \otimes \mathbf{I}_2 \big]_{h_k^{(N)},h_k^{(N)}} = \big[\mathbf{T}_l^{(N-1)}\big]_{h_{k}^{(N-1)},h_{k}^{(N-1)}} = \begin{cases}
-1 & ~~\textrm{if}~~l=k,\\
-\frac{1-\xi_{\beta}}{1-\xi_{\alpha}}&  ~~\textrm{if}~~l\neq k
\end{cases}
\end{equation*}
by (\ref{833-pm-617}) and the induction hypothesis. If $ \mathbf{T}_l^{(N)} = \mathbf{I}_2 \otimes \mathbf{T}_{l-1}^{(N-1)}$, then  for each $k=1,\cdots, N$ and for each $l=2,\cdots, N-1$,
\begin{equation*}
\big[\mathbf{T}_l^{(N)}\big]_{h_k^{(N)},h_k^{(N)}}= \big[\mathbf{I}_2 \otimes \mathbf{T}_{l-1}^{(N-1)} \big]_{h_k^{(N)},h_k^{(N)}} = \big[\mathbf{T}_{l-1}^{(N-1)}\big]_{h_{k-1}^{(N-1)},h_{k-1}^{(N-1)}} = \begin{cases}
-1 & ~~\textrm{if}~~l=k,\\
-\frac{1-\xi_{\beta}}{1-\xi_{\alpha}}&  ~~\textrm{if}~~l\neq k
\end{cases}
\end{equation*}
by (\ref{834-pm-617}) and the induction hypothesis. Finally, we can easily verify that
\begin{eqnarray}
\big[\mathbf{T}_1^{(N)}\big]_{h_N^{(N)},h_N^{(N)}} = -\frac{1-\xi_{\beta}}{1-\xi_{\alpha}} = \big[\mathbf{T}_{N-1}^{(N)}\big]_{h_0^{(N)},h_0^{(N)}}
\end{eqnarray}
by observing the matrices.
\end{itemize}
\end{proof}
\begin{lemma}\label{146-am-520} For each $k=0,1,\cdots,N$ and  $N \geq 2$,
\begin{equation}\label{336-pm-524}
\big[\mathbf{A}_{\sigma}\big]_{h_k,h_k} =
\textrm{sgn}(\sigma)\prod_{j=1}^{k}\bigg(\frac{1-\xi_{j}}{1-\xi_{\sigma(j)}}  \bigg)^{j-1}\prod_{j=k+1}^{N}\bigg(\frac{1-\xi_{j}}{1-\xi_{\sigma(j)}}  \bigg)^{j-2}.
\end{equation}
\end{lemma}
\begin{proof}
If $\sigma = 1$, the statement is trivial. We will show the statement is true for $\sigma \neq 1$. We prove the statement by induction on $N$. When $N=2$, it can be easily shown from (\ref{546pm511}) that (\ref{336-pm-524}) holds  since $\mathbf{A}_{21} =\mathbf{T}_1(1,2)$. Suppose that (\ref{336-pm-524}) holds for $N$. For any $\sigma \in S_{N+1}$, there are $\sigma' \in S_{N}$ and an integer $K\in \{1,\cdots, N+1\}$ such that $\sigma(i) = \sigma'(i)$ for $1\leq i\leq K-1$, $\sigma(K) = N+1$ and $\sigma(i) = \sigma'(i-1)$ for $K+1 \leq i \leq N+1$.  Let $\sigma' = T'_{a_n}\cdots T'_{a_1}$ where $T'_{a_i}$ are simple transpositions in $S_N$ so that $a_1,\cdots,a_n$ are some integers in $ \{1,\cdots,N-1\}$, and let  $T_{a_i}$ be the extension of $T'_{a_i}$ to $\{1,\cdots, N+1\}$ such that $T_{a_i}(k) = T'_{a_i}(k)$ for $k=1,\cdots, N$ and $T_{a_i}(N+1) = N+1$.\\ \\
If $K<N+1$, then
\begin{equation*}
\sigma =T_KT_{K+1}\cdots T_NT_{a_n}\cdots T_{a_1}
\end{equation*}
and if $K=N+1$, then
\begin{equation*}
\sigma =T_{a_n}\cdots T_{a_1}.
\end{equation*}
The matrices $\mathbf{T}_{a_i}$ and $\mathbf{T'}_{a_i}$ corresponding to $T_{a_i}$ and $T'_{a_i}$ satisfy
 \begin{equation*}
 \begin{aligned}
\mathbf{T}_{a_i} =&~ \mathbf{I}_2^{\otimes (a_i-1)} \otimes \mathbf{S}_{\beta\alpha}\otimes \mathbf{I}_2^{\otimes (N-a_i)}\\
=&~\mathbf{I}_2^{\otimes (a_i-1)} \otimes \mathbf{S}_{\beta\alpha}\otimes \mathbf{I}_2^{\otimes (N-a_i-1)}\otimes \mathbf{I}_2\\
=&~\mathbf{T'}_{a_i} \otimes \mathbf{I}_2.
\end{aligned}
\end{equation*}
Let
\begin{equation*}
g_k = \begin{cases}
2h_k-1 & ~~\textrm{for}~k=0,1,\cdots,N,\\
2h_N&~~\textrm{for}~k=N+1.
\end{cases}
\end{equation*}
First, suppose that $K<N+1$ and let us show that for $\mathbf{A}_{\sigma} = \mathbf{T}_K\cdots \mathbf{T}_N\mathbf{T}_{a_n}\cdots \mathbf{T}_{a_1}$,
\begin{equation}\label{336-pm-66}
[\mathbf{A}_{\sigma}]_{g_k,g_k} =
\textrm{sgn}(\sigma)\prod_{j=1}^{k}\bigg(\frac{1-\xi_{j}}{1-\xi_{\sigma(j)}}  \bigg)^{j-1}\prod_{j=k+1}^{N+1}\bigg(\frac{1-\xi_{j}}{1-\xi_{\sigma(j)}}  \bigg)^{j-2}
\end{equation}
for each $k=0,\cdots, N+1$.  Using Lemma 3.1(d) in \cite{Lee-2017} and Lemma \ref{340am629}(\rmnum{1}), we have
\begin{align}
\big[\mathbf{A}_{\sigma}\big]_{g_k,g_k} =&~\big[\mathbf{T}_K\big]_{g_k,g_k}\cdots\big[\mathbf{T}_N\big]_{g_k,g_k}\big[\mathbf{T}_{a_n}\big]_{g_k,g_k}\cdots\big[\mathbf{T}_{a_1}\big]_{g_k,g_k} \nonumber \\
=&~\big[\mathbf{T}_K\big]_{g_k,g_k}\cdots\big[\mathbf{T}_N\big]_{g_k,g_k}\big[\mathbf{T'}_{a_n}\big]_{h_{k},h_{k}}\cdots\big[\mathbf{T'}_{a_1}\big]_{h_{k},h_{k}}  \nonumber \\
=&~\big[\mathbf{T}_K\big]_{g_k,g_k}\cdots\big[\mathbf{T}_N\big]_{g_k,g_k}\big[\mathbf{A}_{\sigma'}\big]_{h_{k},h_{k}}.\label{510-am-66}
\end{align}
If  $k<K$, then  by the induction hypothesis and Lemma \ref{340am629}(\rmnum{2}), (\ref{510-am-66}) becomes
\begin{equation}\label{818-pm-66}
\begin{aligned}
\big[\mathbf{A}_{\sigma}\big]_{g_k,g_k}=&~ \bigg(-\frac{1-\xi_{N+1}}{1-\xi_{\sigma'(K)}}\bigg)\bigg(-\frac{1-\xi_{N+1}}{1-\xi_{\sigma'(K+1)}}\bigg)\cdots  \bigg(-\frac{1-\xi_{N+1}}{1-\xi_{\sigma'(N)}}\bigg) \\
& \hspace{0.5cm} \times
\textrm{sgn}(\sigma')\prod_{j=1}^{k}\bigg(\frac{1-\xi_{j}}{1-\xi_{\sigma'(j)}}  \bigg)^{j-1}\prod_{j=k+1}^{N}\bigg(\frac{1-\xi_{j}}{1-\xi_{\sigma'(j)}}  \bigg)^{j-2},
\end{aligned}
\end{equation}
and (\ref{818-pm-66}) becomes
\begin{equation}
\begin{aligned}
&~(-1)^{N-K+1}\textrm{sgn}(\sigma')\frac{(1-\xi_{N+1})^{N-K+1}}{(1-\xi_{\sigma(K+1)})(1-\xi_{\sigma(K+2)})\cdots(1-\xi_{\sigma(N+1)})} \\
& \hspace{0.5cm} \times\prod_{j=1}^{k}\bigg(\frac{1-\xi_{j}}{1-\xi_{\sigma(j)}}  \bigg)^{j-1}\prod_{j=k+1}^{K-1}\bigg(\frac{1-\xi_{j}}{1-\xi_{\sigma(j)}}\bigg)^{j-2}\prod_{j=K}^N\bigg(\frac{1-\xi_j}{1-\xi_{\sigma(j+1)}}\bigg)^{j-2}\label{336-am-526}
\end{aligned}
\end{equation}
from the relation between $\sigma$ and $\sigma'$. Finally, if we use  $(-1)^{N-K+1}\textrm{sgn}(\sigma')=\textrm{sgn}(\sigma)$, multiply (\ref{336-am-526}) by
\begin{equation*}
\bigg(\frac{1-\xi_{N+1}}{1-\xi_{\sigma(K)}}\bigg)^{K-2} = 1,
\end{equation*}
and rearrange terms, then (\ref{336-am-526}) becomes (\ref{336-pm-66}). If $k \geq K$, then  (\ref{510-am-66}) becomes
\begin{equation}\label{845-pm-66}
\begin{aligned}
\big[\mathbf{A}_{\sigma}\big]_{g_k,g_k}=&~ \bigg(-\frac{1-\xi_{N+1}}{1-\xi_{\sigma'(K)}}\bigg)\cdots
\bigg(-\frac{1-\xi_{N+1}}{1-\xi_{\sigma'(k-1)}}\bigg)(-1)\bigg(-\frac{1-\xi_{N+1}}{1-\xi_{\sigma'(k+1)}}\bigg)\cdots  \bigg(-\frac{1-\xi_{N+1}}{1-\xi_{\sigma'(N)}}\bigg) \\
& \hspace{0.5cm} \times
\textrm{sgn}(\sigma')\prod_{j=1}^{k}\bigg(\frac{1-\xi_{j}}{1-\xi_{\sigma'(j)}}  \bigg)^{j-1}\prod_{j=k+1}^{N}\bigg(\frac{1-\xi_{j}}{1-\xi_{\sigma'(j)}}  \bigg)^{j-2}
\end{aligned}
\end{equation}
by Lemma \ref{340am629}(\rmnum{2}), and  (\ref{845-pm-66}) becomes
\begin{equation}\label{923-pm-66}
\begin{aligned}
&~(-1)^{N-K+1}\textrm{sgn}(\sigma')\frac{(1-\xi_{N+1})^{N-K}}{(1-\xi_{\sigma(K+1)})\cdots(1-\xi_{\sigma(k)})(1-\xi_{\sigma(k+2)})
\cdots(1-\xi_{\sigma(N+1)})} \\
& \hspace{0.5cm} \times\prod_{j=1}^{K-1}\bigg(\frac{1-\xi_{j}}{1-\xi_{\sigma(j)}}  \bigg)^{j-1}\prod_{j=K}^{k}\bigg(\frac{1-\xi_{j}}{1-\xi_{\sigma(j+1)}}\bigg)^{j-1}\prod_{j=k+1}^N\bigg(\frac{1-\xi_j}{1-\xi_{\sigma(j+1)}}\bigg)^{j-2}
\end{aligned}
\end{equation}
from the relation between $\sigma$ and $\sigma'$. Again, if we use  $(-1)^{N-K+1}\textrm{sgn}(\sigma')=\textrm{sgn}(\sigma)$, multiply (\ref{923-pm-66}) by
\begin{equation*}
\bigg(\frac{1-\xi_{N+1}}{1-\xi_{\sigma(K)}}\bigg)^{K-1} = 1,
\end{equation*}
and rearrange terms, then (\ref{923-pm-66}) becomes (\ref{336-pm-66}). \\ \\ Second, suppose that $K=N+1$ and let us show (\ref{336-pm-66}) for each $k=0,\cdots,N+1$  for $\mathbf{A}_{\sigma} = \mathbf{T}_{a_n}\cdots \mathbf{T}_{a_1}$. Using
\begin{equation*}
\big[\mathbf{T}_{a_i}\big]_{g_k,g_k} = \big[\mathbf{T}'_{a_i} \otimes \mathbf{I}_2\big]_{g_k,g_k}=\big[\mathbf{T}'_{a_i}\big]_{h_k,h_k}
\end{equation*}
and the induction hypothesis, we have
\begin{equation*}\label{532-am-629}
\big[\mathbf{A}_{\sigma}\big]_{g_k,g_k} = \big[\mathbf{A}_{\sigma'}\big]_{h_k,h_k}=\textrm{sgn}(\sigma')\prod_{j=1}^{k}\bigg(\frac{1-\xi_{j}}{1-\xi_{\sigma'(j)}}  \bigg)^{j-1}\prod_{j=k+1}^{N}\bigg(\frac{1-\xi_{j}}{1-\xi_{\sigma'(j)}}  \bigg)^{j-2},
\end{equation*}
which  is equal to
\begin{equation*}
\textrm{sgn}(\sigma)\prod_{j=1}^{k}\bigg(\frac{1-\xi_{j}}{1-\xi_{\sigma(j)}}  \bigg)^{j-1}\prod_{j=k+1}^{N+1}\bigg(\frac{1-\xi_{j}}{1-\xi_{\sigma(j)}}  \bigg)^{j-2}
\end{equation*}
because $\sigma(i) = \sigma'(i)$ for $i=1,\cdots, N$ and $\sigma(N+1) = N+1$. This completes the proof.
\end{proof}
Hence, we obtained the explicit formulas for some diagonal terms of $\mathbf{P}_{Y}({X};t)$
\begin{equation}\label{433}
P_{(Y,\nu^{(k)})}(X,\nu^{(k)};t) = \dashint_C\cdots \dashint_C\sum_{\sigma\in {S}_N}\big[\mathbf{A}_{\sigma}\big]_{h_k,h_k}\prod_{i=1}^N\Big(\xi_{\sigma(i)}^{x_i-y_{\sigma(i)}-1}e^{\varepsilon(\xi_i) t}\Big) d\xi_1\cdots d\xi_N,
\end{equation}
where $\big[\mathbf{A}_{\sigma}\big]_{h_k,h_k}$ is given by (\ref{336-pm-524}).
\begin{remark}
In (\ref{433}), if $k=0$ or $N$, then the formula is for the TASEP with a single species.  In these cases, (\ref{336-pm-524}) is equal to
\begin{equation*}
[\mathbf{A}_{\sigma}]_{h_0,h_0} = [\mathbf{A}_{\sigma}]_{1,1} =[\mathbf{A}_{\sigma}]_{2^N,2^N}=[\mathbf{A}_{\sigma}]_{h_N,h_N} = \prod_{(\beta,\alpha)}\bigg(-\frac{1-\xi_{\beta}}{1-\xi_{\alpha}}\bigg)
\end{equation*}
where the product is over all inversions $(\beta,\alpha)$ in $\sigma$.
\end{remark}
\subsection{Some results on the  Vandermonde matrix}
First, we review the generalized Lapalce expansion formula for the determinant in \cite{Murota-book,Zeidler-book} and provide  some new results to prove an algebraic identity which is used in the proof of Theorem \ref{205-am-618}. Let $\mathbf{A}$ be an $N\times N$ matrix. For $I,J \subset \{1,\cdots, N\}$, the submatrix of $\mathbf{A}$ with entries $[\mathbf{A}]_{i,j}$ with $i\in I$ and $j\in J$ is denoted by $\mathbf{A}[I,J]$, and let $I^c = \{1,\cdots,N\}\setminus I$.
\begin{theorem}\cite[p.33]{Murota-book}\label{242-am-69}
For an $N\times N$ matrix $\mathbf{A}$ and $I \subset \{1,\cdots, N\}$, it holds that
\begin{equation*}
\det \mathbf{A} = \sum_{\substack{J \subset \{1,\cdots,N\}, \\ |J| = |I|}}\textrm{sgn}(I,J)\det\mathbf{A}[I,J]\det\mathbf{A}[I^c,J^c].
\end{equation*}
Here, $\textrm{sgn}(I,J)$ is the sign of the permutation
\begin{equation*}
 \left(
               \begin{array}{cccccc}
                 i_1 & \cdots & i_k & i_{k+1} & \cdots & i_N\\
                 j_1 & \cdots & j_k & j_{k+1} & \cdots & j_N \\
               \end{array}
             \right)
\end{equation*}
where $I=\{i_1,\cdots, i_k\}$ and $I^c=\{i_{k+1},\cdots, i_N\}$ with $i_1<\cdots<i_k$ and $i_{k+1}<\cdots<i_N$, and $J=\{j_1,\cdots, j_k\}$ and $J^c=\{j_{k+1},\cdots, j_N\}$ with $j_1<\cdots<j_k$ and $j_{k+1}<\cdots<j_N$.
\end{theorem}
The next result generalizes Lemma 3.3 in \cite{Lee-2017}.
\begin{corollary}\label{1148-pm-630}
Let $\mathbf{A}$ be the $N \times N$ Vandermonde matrix with $[\mathbf{A}]_{i,j} = \xi_j^{i-1}$ and $I=\{N-k+1,\cdots, N\}$. Then,
\begin{equation}
\sum_{\substack{J\subset\{1,\cdots,N\}, \\|J|=k}}\textrm{sgn}\big(I,J\big)\prod_{i\in J}(\xi_i-1)^{N-k}\prod_{\substack{i<j,\\i,j\in J}}(\xi_j-\xi_i)\prod_{\substack{i<j,\\ i,j \in J^c}}(\xi_j-\xi_i) = \prod_{1\leq i<j\leq N}(\xi_j - \xi_i). \label{611-am-69}
\end{equation}
\end{corollary}
\begin{proof}
The left hand side of (\ref{611-am-69}) is equal to
\begin{equation}\label{747}
\det\left[
  \begin{array}{cccc}
    1 & 1 & \cdots & 1  \\
    \xi_1 & \xi_2 & \cdots & \xi_N  \\
    \vdots &  \vdots &  & \vdots \\
    \xi_1^{N-k-1} &\xi_2^{N-k-1}  & \cdots&   \xi_N^{N-k-1}\\
    (\xi_1-1)^{N-k} & (\xi_2-1)^{N-k} & \cdots & (\xi_N-1)^{N-k}  \\
    \xi_1(\xi_1-1)^{N-k} & \xi_2(\xi_2-1)^{N-k} & \cdots & \xi_N(\xi_N-1)^{N-k}  \\
     \vdots& \vdots  &  &  \vdots  \\
    \xi_1^{k-1}(\xi_1-1)^{N-k} & \xi_2^{k-1}(\xi_2-1)^{N-k} & \cdots & \xi_N^{k-1}(\xi_N-1)^{N-k}  \\
  \end{array}
\right]
\end{equation}
by  Theorem \ref{242-am-69}. We perform row operations (adding multiples of the first $(N-k)$ rows to the $(N-k+1)^{\textrm{th}}$ row) which make the $(N-k+1)^{\textrm{th}}$ row $(\xi_1^{N-k}~~\xi_2^{N-k}~~\cdots~~\xi_N^{N-k})$
without changing the determinant. We repeat this procedure to change the form  of (\ref{747}) to the Vandermonde determinant.
\end{proof}
\begin{lemma}\label{1144-pm-630}
For $N\geq 2$,
\begin{equation*}
\sum_{\sigma \in S_N}\textrm{sgn}(\sigma)\prod_{j=2}^N\frac{1}{(\xi_{\sigma(N+1-j)}-1)^{j-1}} = \prod_{j=1}^N\frac{1}{(\xi_j -1)^{N-1}}\prod_{1\leq i<j \leq N}(\xi_j - \xi_i).
\end{equation*}
\end{lemma}
\begin{proof}
We obtain the result from
\begin{equation*}
\begin{aligned}
\prod_{j=2}^N\frac{1}{(\xi_{\sigma(N+1-j)}-1)^{j-1}} =\frac{(\xi_{\sigma(2)}-1)(\xi_{\sigma(3)}-1)^2\cdots(\xi_{\sigma(N)}-1)^{N-1}}{ \prod_{j=1}^N{(\xi_j -1)^{N-1}}}
\end{aligned}
\end{equation*}
and
\begin{equation*}
\begin{aligned}
& \sum_{\sigma \in S_N}\textrm{sgn}(\sigma)(\xi_{\sigma(2)}-1)(\xi_{\sigma(3)}-1)^2\cdots(\xi_{\sigma(N)}-1)^{N-1} \\
=&~
\det\left[
  \begin{array}{cccccc}
    1 & (\xi_1-1) & (\xi_1-1)^2 & \cdots & (\xi_1-1)^{N-2} & (\xi_1-1)^{N-1} \\
    1 & (\xi_2-1) & (\xi_2-1)^2 & \cdots & (\xi_2-1)^{N-2} & (\xi_2-1)^{N-1} \\
    \vdots & \vdots & \vdots &  & \vdots& \vdots\\
    1 & (\xi_N-1) & (\xi_N-1)^2 & \cdots & (\xi_N-1)^{N-2}& (\xi_N-1)^{N-1} \\
  \end{array}
\right]\\
=&~
\det\left[
  \begin{array}{cccccc}
    1 & \xi_1 & \xi_1^2 & \cdots & \xi_1^{N-2} & \xi_1^{N-1} \\
    1 & \xi_2 & \xi_2^2 & \cdots & \xi_2^{N-2} & \xi_2^{N-1} \\
    \vdots & \vdots & \vdots &  & \vdots& \vdots\\
    1 & \xi_N & \xi_N^2 & \cdots & \xi_N^{N-2}& \xi_N^{N-1} \\
  \end{array}
\right] = \prod_{1\leq i<j \leq N}(\xi_j - \xi_i).
\end{aligned}
\end{equation*}
\end{proof}
\section{Proofs of theorems}
We will compute
\begin{equation}
\mathbb{P}_{(Y,\nu^{(k)})}(E_{t,k,x}) = \sum_{\substack{x_k<x_{k+1}<\cdots<x_N,\\x_1=x,x_2=x+1,\cdots,x_k = x+k-1}} P_{(Y,\nu^{(k)})}(X,\nu^{(k)};t).\label{1219-am-520}
\end{equation}
The integrand hidden  in the form of $P_{(Y,\nu^{(k)})}(X,\nu^{(k)};t)$ in (\ref{1219-am-520}) has a factor
\begin{equation}\label{105-am-67}
\sum_{\sigma \in S_N}\sum_{i_1,\cdots,i_{N-k} = 1}^{\infty} \big[\mathbf{A}_{\sigma}\big]_{h_k,h_k}\xi_{\sigma(2)}\xi_{\sigma(3)}^{2}\cdots\xi_{\sigma(k)}^{k-1}\xi_{\sigma(k+1)}^{k-1+i_1}\cdots\xi_{\sigma(N)}^{k-1+i_1+\cdots+i_{N-k}}
\end{equation}
with positive integers $i_1,\cdots,i_{N-k}$. Since  $|\xi_i|<1$ for each $i$, the multiple geometric series in (\ref{105-am-67}) converge, hence (\ref{105-am-67}) becomes
\begin{equation*}
\sum_{\sigma \in S_N}\big[\mathbf{A}_{\sigma}\big]_{h_k,h_k}\frac{\xi_{\sigma(2)}\xi_{\sigma(3)}^{2}\cdots \xi_{\sigma_{(N)}}^{N-1}}{(1-\xi_{\sigma(k+1)}\cdots\xi_{\sigma(N)})(1-\xi_{\sigma(k+2)}\cdots\xi_{\sigma(N)})\cdots (1-\xi_{\sigma(N)})}.
\end{equation*}
\begin{lemma}\label{1141-pm-630}
Let $N \geq 2$. For each $k=0,1,\cdots, N$,
\begin{equation}\label{604-pm-67}
\begin{aligned}
&\sum_{\sigma \in S_N}\big[\mathbf{A}_{\sigma}\big]_{h_k,h_k}\frac{\xi_{\sigma(2)}\xi_{\sigma(3)}^{2}\cdots \xi_{\sigma_{(N)}}^{N-1}}{(1-\xi_{\sigma(k+1)}\cdots\xi_{\sigma(N)})(1-\xi_{\sigma(k+2)}\cdots\xi_{\sigma(N)})\cdots (1-\xi_{\sigma(N)})} \\
=&~(1-\xi_1)\cdots(1-\xi_k)\prod_{1\leq i<j\leq N}\frac{\xi_j-\xi_i}{1-\xi_i}\prod_{i=1}^N\frac{1}{1-\xi_i}.
\end{aligned}
\end{equation}
\end{lemma}
\begin{remark}
When $k=0~\textrm{or}~N$, the identity (\ref{604-pm-67}) is for the TASEP. If $k=0$, the identity is for the probability that $x_1 \geq  x$ at time $t$ and (\ref{604-pm-67}) is the same as (1.6) with $p=1$ in \cite{Tracy-Widom-2008}. If $k=N$, (\ref{604-pm-67}) is for the probability that $x_1=x,x_2=x+1,\cdots, x_N=x+N-1$ at time $t$.
\end{remark}
Dividing both sides of (\ref{604-pm-67}) by $\prod_{j=1}^k(1-\xi_j)^{j-1}\prod_{j=k+1}^N(1-\xi_j)^{j-2}$ and then substituting $1/\xi_{N-i+1}$ for $\xi_i $, we see that an equivalent version of (\ref{604-pm-67}) is
\begin{equation}\label{615-pm-67}
\begin{aligned}
&\sum_{\sigma \in S_N}\textrm{sgn}(\sigma)\frac{1}{\prod_{j=2}^k\big(\xi_{\sigma(N+1-j)}-1\big)^{j-1}\prod_{j=k+1}^{N}\big(\xi_{\sigma(N+1-j)}-1\big)^{j-2}}\\
&~\hspace{1cm}\times \frac{\xi_{\sigma(N-k-1)}\xi_{\sigma(N-k-2)}^{2}\cdots \xi_{\sigma{(1)}}^{N-k-1}}{\big(\xi_{\sigma(N-k)}\cdots\xi_{\sigma(1)}-1\big)\big(\xi_{\sigma(N-k-1)}\cdots\xi_{\sigma(1)}-1\big)\cdots \big(\xi_{\sigma(1)}-1\big)} \\
=&~\prod_{i=1}^N\frac{1}{(\xi_i-1)^{N-1}}\prod_{1\leq i<j\leq N}(\xi_j-\xi_i).
\end{aligned}
\end{equation}
We will prove (\ref{615-pm-67}) for each $k=0,1,\cdots, N$. The main idea of the proof is essentially the same as the idea in the proof of (47) in \cite{Lee-2017}.
\begin{proof}[Proof of the identity (\ref{615-pm-67})]
If $k=0$, then the identity (\ref{615-pm-67}) is equal to (50) in \cite{Lee-2017} or essentially equal to (1.6) with $p=1$ in \cite{Tracy-Widom-2008}. If $k=1$, then the identity (\ref{615-pm-67}) is equal to (47) in \cite{Lee-2017}. Let us show the identity for $k=2,\cdots, N$. The left hand side of (\ref{615-pm-67}) is an antisymmetric function of $\xi_1,\cdots, \xi_{N}$ because it is an anti-symmetrized  sum, so it can be written as
\begin{equation*}
G(\xi_1,\cdots, \xi_{N})\times \prod_{1\leq i<j\leq N}(\xi_j-\xi_i)
\end{equation*}
for some symmetric function $G(\xi_1,\cdots, \xi_{N})$. We will  show that
\begin{equation*}
G(\xi_1,\cdots, \xi_{N})=\prod_{i=1}^N\frac{1}{(\xi_i-1)^{N-1}}.
\end{equation*}
 Let us fix a subset $\Lambda = \{\alpha_1,\cdots, \alpha_k\} \subset \{1,\cdots, N\}$. Let $\sigma'$ be a permutation mapping from $\{1,\cdots,N-k\}$ onto $\Lambda^c=\{1,\cdots,N\}\setminus \Lambda$ and $S'_{N-k}$ be the symmetric group of $\sigma'$. Let $\sigma''$ be a permutation mapping from $\{N-k+1,\cdots,N\}$ onto $\Lambda$ and $S''_{k}$ be the symmetric group of $\sigma''$. Let $\sigma_{\Lambda}$ be a permutation in $S_N$ such that
\begin{equation}\label{832-pm-68}
\begin{aligned}
&\sigma_{\Lambda}(i) = \sigma'(i)&\textrm{for}~i=1,\cdots,N-k~~\textrm{and}\\
&\sigma_{\Lambda}(i) = \sigma''(i)&\textrm{for}~i=N-k+1,\cdots,N.
\end{aligned}
\end{equation}
Then, the left hand side of (\ref{615-pm-67}) is equal to
\begin{equation}\label{831-pm-68}
\begin{aligned}
&\sum_{\Lambda}\sum_{\sigma_{\Lambda} \in S_N}\textrm{sgn}(\sigma_{\Lambda})\frac{1}{\prod_{j=2}^k\big(\xi_{\sigma_{\Lambda}(N+1-j)}-1\big)^{j-1}\prod_{j=k+1}^{N}\big(\xi_{\sigma_{\Lambda}(N+1-j)}-1\big)^{j-2}}\\
&~\hspace{1cm}\times \frac{\xi_{\sigma_{\Lambda}(N-k-1)}\xi_{\sigma_{\Lambda}(N-k-2)}^{2}\cdots \xi_{\sigma_{\Lambda}{(1)}}^{N-k-1}}{\big(\xi_{\sigma_{\Lambda}(N-k)}\cdots\xi_{\sigma_{\Lambda}(1)}-1\big)\big(\xi_{\sigma_{\Lambda}(N-k-1)}\cdots\xi_{\sigma_{\Lambda}(1)}-1\big)\cdots \big(\xi_{\sigma_{\Lambda}(1)}-1\big)}.
\end{aligned}
\end{equation}
Using (\ref{832-pm-68}) and the notation in Theorem \ref{242-am-69}, we have a relation
\begin{equation*}
\textrm{sgn}(\sigma_{\Lambda})=\textrm{sgn}\big(\{1,\cdots,N-k\},\Lambda^c\big)\times \textrm{sgn}(\sigma')\times \textrm{sgn}(\sigma''),
\end{equation*}
hence, (\ref{831-pm-68}) is equal to
\begin{equation}\label{833-pm-68}
\begin{aligned}
&\sum_{\Lambda}\sum_{\sigma' \in S'_{N-k}}\sum_{\sigma'' \in S''_k}\textrm{sgn}\big(\{1,\cdots,N-k\},\Lambda^c\big)\times \textrm{sgn}(\sigma')\times \textrm{sgn}(\sigma'') \\
&~\hspace{1cm}\times\frac{1}{\prod_{j=2}^k\big(\xi_{\sigma''(N+1-j)}-1\big)^{j-1}\prod_{j=k+1}^{N}\big(\xi_{\sigma'(N+1-j)}-1\big)^{j-2}}\\
&~\hspace{1cm}\times \frac{\xi_{\sigma'(N-k-1)}\xi_{\sigma'(N-k-2)}^{2}\cdots \xi_{\sigma'{(1)}}^{N-k-1}}{\big(\xi_{\sigma'(N-k)}\cdots\xi_{\sigma'(1)}-1\big)\big(\xi_{\sigma'(N-k-1)}\cdots\xi_{\sigma'(1)}-1\big)\cdots \big(\xi_{\sigma'(1)}-1\big)}.
\end{aligned}
\end{equation}
Summing over all $\sigma'$ in $S'_{N-k}$,
\begin{equation*}
\begin{aligned}
&\sum_{\sigma' \in S'_{N-k}}\textrm{sgn}(\sigma')\frac{1}{\big(\xi_{\sigma'(N-k)}-1\big)^{k-1}\cdots\big(\xi_{\sigma'(1)}-1\big)^{k-1}}\\
&~\hspace{1cm}\times\frac{1}{\big(\xi_{\sigma'(N-k-1)}-1\big)\big(\xi_{\sigma'(N-k-2)}-1\big)^2\cdots\big(\xi_{\sigma'(1)}-1\big)^{N-k-1}}\\
&~\hspace{1cm}\times \frac{\xi_{\sigma'(N-k-1)}\xi_{\sigma'(N-k-2)}^{2}\cdots \xi_{\sigma'{(1)}}^{N-k-1}}{\big(\xi_{\sigma'(N-k)}\cdots\xi_{\sigma'(1)}-1\big)\big(\xi_{\sigma'(N-k-1)}\cdots\xi_{\sigma'(1)}-1\big)\cdots \big(\xi_{\sigma'(1)}-1\big)} \\
=&~\prod_{i\in \Lambda^c}\frac{1}{(\xi_i-1)^{k-1}}\prod_{i\in \Lambda^c}\frac{1}{(\xi_i-1)^{N-k}}\prod_{\substack{i<j,\\i,j\in \Lambda^c}}(\xi_j-\xi_i)
\end{aligned}
\end{equation*}
by (51) in \cite{Lee-2017} with $\sigma' \in S'_{N-k}$ instead of $\sigma\in S_N$. Summing over $\sigma'' \in S''_k$,
\begin{equation*}
\sum_{\sigma'' \in S''_k}\textrm{sgn}(\sigma'')\prod_{j=2}^k\frac{1}{\big(\xi_{\sigma''(N+1-j)}-1\big)^{j-1}} = \prod_{\substack{i<j,\\ i,j \in \Lambda}}(\xi_j-\xi_i)\prod_{i\in \Lambda}\frac{1}{(\xi_i-1)^{k-1}}
\end{equation*}
by Lemma \ref{1144-pm-630}. Hence, (\ref{833-pm-68}) becomes
\begin{equation*}
\begin{aligned}
&\prod_{i=1}^N\frac{1}{(\xi_i-1)^{N-1}}\times \sum_{\Lambda}\textrm{sgn}\big(\{1,\cdots,N-k\},\Lambda^c\big)\prod_{i\in \Lambda}(\xi_i-1)^{N-k}\prod_{\substack{i<j,\\i,j\in \Lambda^c}}(\xi_j-\xi_i)\prod_{\substack{i<j,\\ i,j \in \Lambda}}(\xi_j-\xi_i)\\
=&~\prod_{i=1}^N\frac{1}{(\xi_i-1)^{N-1}}\prod_{1\leq i<j\leq N}(\xi_j - \xi_i)
\end{aligned}
\end{equation*}
by Corollary \ref{1148-pm-630} because $\textrm{sgn}(I,J) = \textrm{sgn}(I^c,J^c)$.
\end{proof}
\begin{proof}[Proof of Theorem \ref{205-am-618}]
Theorem \ref{205-am-618} is immediately proved by Lemma \ref{1141-pm-630}.
\end{proof}
The proof of Theorem \ref{1143-am-629} is similar to the proof of Theorem 1.2 in \cite{Lee-2017}.
\begin{proof}[Proof of Theorem \ref{1143-am-629}]
 If we apply  $Y=(1,\cdots, N)$ to Theorem \ref{205-am-618}, (\ref{225-am-618}) becomes
\begin{eqnarray*}
\mathbb{P}_{(Y,\nu^{(k)})}(E_{t,k,x}) &=&~ \dashint_C\cdots\dashint_C\prod_{1\leq i<j\leq N}(\xi_j-\xi_i)\prod_{i=1}^N\frac{1}{(1-\xi_i)^{N-1}}\prod_{i=1}^N\bigg(\xi_i^{x-N-1}e^{\varepsilon(\xi_i)t}\bigg)\\
~& &\hspace{1cm}\times\bigg(\prod_{j=1}^k(1-\xi_j)^{j-1}\prod_{j=k+1}^N(1-\xi_j)^{j-2}\bigg)\xi_1^{N-1}\xi_2^{N-2}\cdots\xi_{N-1}d\xi_1\cdots d\xi_N \nonumber\\
&=&~\frac{1}{N!}\dashint_C\cdots\dashint_C\prod_{1\leq i<j\leq N}(\xi_j-\xi_i)\prod_{i=1}^N\frac{1}{(1-\xi_i)^{N-1}}\prod_{i=1}^N\bigg(\xi_i^{x-N-1}e^{\varepsilon(\xi_i)t}\bigg) \nonumber\\
~& &\hspace{1cm}\times\Bigg[\sum_{\sigma \in S_N} \textrm{sgn}(\sigma)\bigg(\prod_{j=1}^k(1-\xi_{\sigma(j)})^{j-1}\prod_{j=k+1}^N(1-\xi_{\sigma(j)})^{j-2}\bigg)\\
~& &\hspace{1cm}\times\xi_{\sigma(1)}^{N-1}\xi_{\sigma(2)}^{N-2}\cdots\xi_{\sigma(N-1)}\Bigg]d\xi_1\cdots d\xi_N.
\end{eqnarray*}
Using Lemma 3.4 in \cite{Lee-2017} as in the proof of Theorem 1.2 in \cite{Lee-2017},  we obtain
\begin{equation*}
\begin{aligned}
& \sum_{\sigma \in S_N} \textrm{sgn}(\sigma)\bigg(\prod_{j=1}^k(1-\xi_{\sigma(j)})^{j-1}\prod_{j=k+1}^N(1-\xi_{\sigma(j)})^{j-2}\bigg)\xi_{\sigma(1)}^{N-1}\xi_{\sigma(2)}^{N-2}
\cdots\xi_{\sigma(N-1)}\\
& \hspace{2cm}=  (-1)^{N(N-1)/2}\prod_{1\leq i<j\leq N}(\xi_j - \xi_i)
\end{aligned}
\end{equation*}
and  finally obtain the formula (\ref{534-pm-517}).
\end{proof}
\vspace{1cm}
\noindent \textbf{Acknowledgement}\\
The author is grateful to anonymous referees for valuable comments. This work was supported by the social policy grant of Nazarbayev University.\\

\end{document}